\documentclass[12pt,reqno]{amsart}
\usepackage{amsmath, amsfonts, amssymb, amsthm, mathrsfs}
\def\fM{\mathfrak M}
\def\fm{\mathfrak m}

\def\Z{\mathbb Z}
\def\T{\mathbb T}

\def\N{\mathbb N}
\def\Q{\mathbb Q}
\def\R{\mathbb R}
\def\E{\mathbb E}

\def\cR{\mathcal R}

\def\cJ{\mathcal J}
\def\cK{\mathcal K}
\def\cZ{\mathcal Z}
\def\cG{\mathcal G}
\def\cX{\mathcal X}
\def\cY{\mathcal Y}
\def\cU{\mathcal U}
\def\cV{\mathcal V}

\def\fX{\mathfrak X}
\def\fC{\mathfrak C}

\def\fm{\mathfrak m}

\def\fd{\mathfrak d}
\def\fe{\mathfrak e}

\def\bfa{\mathbf{a}}
\def\bfb{\mathbf{b}}
\def\sB{\mathscr{B}}
\def\sK{\mathscr{K}}

\def\1{{\bf 1}}

\def\pmod #1{\ ({\rm{mod}}\ #1)}

\theoremstyle{plain}
\newtheorem{theorem}{Theorem}
\newtheorem{lemma}{Lemma}
\newtheorem{corollary}{Corollary}

\newtheorem{proposition}{Proposition}
\theoremstyle{definition}

\theoremstyle{remark}
\newtheorem{remark}{Remark}

\begin{document}
\title{Ergodic recurrence and bounded gaps between primes}

\begin{abstract}
Let $(\cX,\sB_\cX,\mu,T)$ be a measure-preserving probability system with $T$ is invertible. 
Suppose that $A\in \sB_\cX$ with $\mu(A)>0$ and $\epsilon>0$. For any $m\geq 1$,
there exist infinitely many primes $p_0,p_1,\ldots,p_m$ with $p_0<\cdots<p_m$ such that
$$
\mu(A\cap T^{-(p_i-1)}A)\geq \mu(A)^2-\epsilon
$$
for each $0\leq i\leq m$ and
$$
p_m-p_0<C_m,
$$
where $C_m>0$ is a constant only depending on $m$, $A$ and $\epsilon$.

\end{abstract}
\author{Hao Pan}
\address{Department of Mathematics, Nanjing University, Nanjing 210093,
People's Republic of China}
\email{haopan79@zoho.com}

\subjclass[2010]{}
\maketitle

\section{Introduction}
\setcounter{lemma}{0}\setcounter{theorem}{0}\setcounter{proposition}{0}\setcounter{corollary}{0}\setcounter{remark}{0}
\setcounter{equation}{0}

Suppose that $(\cX,\sB_\cX,\mu)$ is a probability measure space. Let $T:\,\cX\to\cX$ be a measure-preserving transformation on $\cX$, i.e., $\mu(A)=\mu(T^{-1}A)$ for each $A\in\sB_\cX$. The classical Poincare recurrence theorem (cf. \cite[Theorem 2.11]{EW11}) says that for any $A\in \sB_\cX$ with $\mu(A)>0$, there exist infinitely many integers $n\geq 1$ such that $\mu(A\cap T^{-n}A)>0$.

Assume that $T$ is an invertible transformation. In 1934, Khintchine \cite{Kh34} considered the set of recurrence
$$
S_{A,\epsilon}=\{n\in\N:\,\mu(A\cap T^{-n}A)\geq\mu(A)^2-\epsilon\}.
$$
Khintchine proved that for any $A\in \sB_\cX$ with $\mu(A)>0$ and any $\epsilon>0$, $S_{A,\epsilon}$ has a bounded gap, i.e., there exists a constant $L_0>0$ (only depending on $A$ and $\epsilon$) such that
$$
S_{A,\epsilon}\cap[x,x+L_0]\neq\emptyset
$$
for any $x\geq 1$. Notice that the bounded gaps can be arbitrarily large for the varied measure-preserving systems. For example, for any positive integer $m$, consider the cyclic group $\Z_m=\Z/m\Z$ with the discrete probability measure. Let $T:\,x\mapsto x+1$ and $A=\{0\}\subseteq\Z_m$. Clearly $A\cap T^{-n}A\neq\emptyset$ if and only if $n$ is a multiple of $m$. So the bounded gap of $S_{A,\epsilon}$ is always $m$ for any $0<\epsilon<m^{-2}$.
For the further extensions of Khintchine's theorem, the readers may refer to \cite{BHK05,Fr08,Ch11,CFH11}.

On the other hand, a recent breakthrough on number theory is about the bounded gaps between consecutive primes. In \cite{Zh14},
with help the Goldston-Pintz-Y{\i}d{\i}r{\i}m \cite{GPY09} sieve method and a variant of the Bombieri-Vinogradov theorem, Zhang showed for the first time that the gap between consecutive two primes can be bounded by a constant infinitely often. In fact, Zhang proved that
$$
\liminf_{\substack{p,q\to\infty\\ p<q\text{ are primes}}}(q-p)\leq7\times 10^7.
$$
Subsequently, using a multi-dimensional sieve method, Maynard \cite{Ma15} greatly improved the bound $7\times 10^7$ to $600$. Nowadays, the best bound is $246$ \cite{Pol14}. In fact,  with help of the multi-dimensional sieve method,
Maynard and Tao independently showed that for any $m\geq 1$, the gaps between consecutive $m$ primes also can be bounded by a
constant  infinitely 
many times.
That is, there exist infinitely many primes $p_0,p_1,\ldots,p_m$ with $p_0<\cdots<p_m$ such that
$$
p_m-p_0<C_m,
$$
where $C_m>0$ is a constant only depending on $m$. Subsequently, the Maynard-Tao theorem was extended to the primes of some special types \cite{Po14, CPS15, LP15}.
There is a nice survey on Zhang's theorem and Maynard-Tao's theorem written by Granville \cite{Ga15}

Notice that Khintchine's theorem, Zhang's theorem and Maynard-Tao's theorem are all concerning the bounded gaps.
It is natural to ask whether we can establish a connection between those theorems.
The purpose of this paper is to give a Khintchine type extension to the Maynard-Tao theorem.

Let us consider those $n$ in $S_{A,\epsilon}$ which are shifted primes, i.e., a prime minus one.
For $A\in \sB_\cX$ with $\mu(A)>0$ and $\epsilon>0$, let
$$
\Lambda_{A,\epsilon}=\{p\text{ prime}:\,\mu(A\cap T^{-(p-1)}A)\geq\mu(A)^2-\epsilon\}.
$$
\begin{theorem}\label{mainT} Let $(\cX,\sB_\cX,\mu,T)$ be a measure-preserving probability system with $T$ is invertible. 
Suppose that $A\in \sB_\cX$ with $\mu(A)>0$ and $\epsilon>0$.
For any $m\geq 1$, there exist infinitely many primes $$p_0,p_1,\ldots,p_m\in\Lambda_{A,\epsilon}$$ with $p_0<\cdots<p_m$ such that
$$
p_m-p_0<C_m,
$$
where $C_m>0$ is a constant only depending on $m$, $A$ and $\epsilon$.
\end{theorem}
In particular, we may find infinitely many pairs of primes $p,q$ with $p<q$ such that
$$
\mu(A\cap T^{-(p-1)}A),\ \mu(A\cap T^{-(q-1)}A)>0
$$
and the gap $q-p$ is bounded by a constant $C$.

Let us see a combinatorial consequence of Theorem \ref{mainT}. For any $E\subseteq\N$, define the upper Banach density of $E$ 
$$
\overline{d}_B(E):=\limsup_{\substack{N-M\to+\infty\\ N\geq M\geq 0}}\frac{|E\cap[M,N]|}{N-M+1}.
$$
In \cite{Sa78}, Sark\"ozy proved that if $\overline{d}_B(E)>0$, there exist infinitely many prmes $p$ such that
$$
p-1\in E-E,
$$
where $E-E=\{x-y:\,x,y\in E\}$. Further, Bergelson and Lesigne \cite{BL08} showed 
$\{p-1:\,p\text{ is prime}\}$ is an enhanced van der Corput set.

Now with help of the well-known Furstenberg corresponding principle (cf. \cite[Lemma 2.5]{Fu81}), 
we can obtain
\begin{corollary}
Suppose that $E$ is a subset of $\N$ with $\overline{d}_B(E)>0$ and $\epsilon>0$.
Let
$$
\Lambda_{E,\epsilon}^*=\big\{p\text{ prime}:\,\overline{d}_B\big(E\cap(p-1+E)\big)\geq\overline{d}_B(E)^2-\epsilon\big\}.
$$
Then for any $m\geq 1$, there exist infinitely many primes $p_0,p_1,\ldots,p_m\in \Lambda_{E,\epsilon}^*$ with $p_0<\cdots<p_m$ such that $p_m-p_1$ is bounded by a constant $C_m$.
\end{corollary}
Suppose that $N$ is sufficiently large and $W=W_0\prod_{p\leq w}p$, where $w$ very slowly tends to infinity as $N\to+\infty$. Let $n\sim N$ mean $N\leq n\leq 2N$.
In view of the Maynard sieve method, we need to compute the sum
\begin{equation}\label{varpiOmeganMuATnA}
\sum_{\substack{n\sim N\\ n\equiv b\pmod{W}}}\varpi(n+h)\Omega_n\cdot\mu(A\cap T^{-(n+h-1)}A),
\end{equation}
where $\Omega_n\geq 0$ is  some weight and
$$
\varpi(n)=\begin{cases}\log n,&\text{if }n\text{ is prime},\\
0,&\text{otherwise}.\end{cases}
$$
As we shall see later, though seemingly it is not easy to give an asymptotic formula for (\ref{varpiOmeganMuATnA}), we can obtain a suitable lower bound.
Define
$$
\1_A(x)=\begin{cases}1,&\text{if }x\in A,\\
0,&\text{otherwise}.
\end{cases}
$$
Our strategy is to write
$$
\1_A(x)=f_1(x)+f_2(x)
$$
under the assumption $T$ is ergodic,
where $f_1$ belongs to $L^2(\cX,\sK,\mu)$, which is the closed $L^2$-subspace generated by all eigenfunctions of $T$, and $f_2$ is orthogonal to $L^2(\cX,\sK,\mu)$.

In Section 3, we shall show that
$$
\lim_{N\to\infty}\frac{}{}\sum_{\substack{n\sim N\\ n\equiv b\pmod{W}}}\varpi(n+h)\Omega_n\int_{\cX}f_2\cdot T^{n+h-1}f_2d\mu=0.
$$
And in Section 4, we can give a lower bound for
$$
\lim_{N\to\infty}\sum_{\substack{n\sim N\\ n\equiv b\pmod{W}}}\varpi(n+h)\Omega_n\int_{\cX}f_1\cdot T^{n+h-1}f_1d\mu,
$$
provided that $W,b,h$ satisfy some additional assumptions. In fact, we shall transfer this problem to studying an ergodic transformation on a torus, and use some basic techniques from the Diophantine approximation involving primes.
Of course, firstly we need to apply the Maynard sieve method to the exponential sum
$$
\lim_{N\to\infty}\sum_{\substack{n\sim N\\ n\equiv b\pmod{W}}}\varpi(n+h)\Omega_n\cdot
e\bigg(n\cdot\bigg(\frac{a}{q}+\theta\bigg)\bigg)
$$
in Section 2, where as usual let $e(x)=\exp(2\pi\sqrt{-1}x)$ for $x\in\R$.
Finally, with help of the ergodic decomposition theorem, we shall conclude the proof of Theorem \ref{mainT} in Section 5.

Throughout the whole paper, let $f(x)\ll g(x)$ mean $f(x)=O\big(g(x)\big)$ as $x$ tends to $+\infty$, i.e., $|f(x)|\leq C|g(x)|$ for some constant $C>0$ whenever $x$ is sufficiently large.
And $\ll_\epsilon$ means the implied constant in $\ll$ only depends on $\epsilon$. 
Let $\tau$ and $\phi$ denote the divisor function and the Euler totient function respectively. In particular, if $n$ is an integer, let $\mu(n)$ denote the value of the M\"obius function at $n$, rather than the measure on $\cX$.

\section{Maynard's sieve method for the exponential sums over primes}
\setcounter{lemma}{0}\setcounter{theorem}{0}\setcounter{proposition}{0}\setcounter{corollary}{0}\setcounter{remark}{0}
\setcounter{equation}{0}

Let $N$ be a sufficiently large integer and $R=N^{\frac1{1000}}$. 
Suppose that $w$ is a large integer with $w\leq\log\log\log N$. Let
\begin{equation}\label{Ww}
W=W_0\prod_{p\leq w\text{ prime}}p
\end{equation}
where $W_0$ is an fixed integer to be chosen later.

For distinct integers $h_0,h_1,\ldots,h_{k}$, we say $\{h_0,h_1,\ldots,h_{k}\}$ is an {\it admissible} set provided that for any prime $p$, there exists $1\leq a\leq p$ satisfying 
$$
a\not\equiv h_j\pmod{p}
$$ 
for each $0\leq j\leq k$.
We may construct an admissible set whose elements are all the multiple of $W_0$. In fact, let $$
h_j=jW_0\prod_{p\leq k+1}p$$ for $j=0,1,\ldots,k$. It is easy to see that $\{h_0,h_1,\ldots,h_{k}\}$ is an admissible set.

Suppose that  $\{h_0,h_1,\ldots,h_{k}\}$ is an admissible set. We may find $1\leq b\leq W$ such that
\begin{equation}\label{bhjp}
b\not\equiv-h_j\pmod{p}
\end{equation} 
for any prime $p\leq w$ and each $0\leq j\leq k$. Further, clearly we may assume that 
$b\equiv 1\pmod{W_0}$.
Also, assume that $w$ is sufficiently large such that the prime factors of $h_j-h_i$ is not greater than $w$ for each $0\leq i<j\leq k$. So if $n\equiv b\pmod{W}$, then $W,n+h_0,\ldots,n+h_k$ are co-prime.

Throughout this paper, we always make the following assumptions:\medskip

\noindent(I) $\{h_0,h_1,\ldots,h_k\}$ is an admissible set;\medskip

\noindent(II) $W,h_0,h_1,\ldots,h_k$ are divisible by $W_0$;\medskip

\noindent(III) $W,n+h_0,\ldots,n+h_k$ are co-prime for any $n\equiv b\pmod{W}$;\medskip

\noindent(IV) $b\equiv 1\pmod{W_0}$.\medskip

Suppose that $F(t_0,t_1,\ldots,t_k)$ is a smooth function over $\R^{k+1}$ whose support lies on
\begin{equation}\label{area}
\{(t_0,t_1,\ldots,t_k):\,t_0,\ldots,t_k\geq 0,\ t_0+\cdots+t_k\leq 1\}.
\end{equation}
Define
$$
\lambda_{d_0,d_1,\ldots,d_k}(F)=F\bigg(\frac{\log d_0}{\log R},\frac{\log d_1}{\log R},\ldots,\frac{\log d_k}{\log R}\bigg)\prod_{j=0}^k\mu(d_j)
$$
and let
$$
\Omega_n(F)=\bigg(\sum_{\substack{d_i\mid n+h_i\\ 0\leq i\leq k}}\lambda_{d_0,d_1,\ldots,d_k}(F)\bigg)^2.
$$
The following lemma is the main ingredient of Maynard's sieve method.
\begin{lemma}\label{maynard}
Suppose that $F_1(x_0,\ldots,x_{k})$ and $F_2(x_0,\ldots,x_{k})$ are two smooth functions whose supports lie on the area (\ref{area}). Then
\begin{align}\label{maynardid}
&\sum_{\substack{d_0,\ldots,d_{k},e_0,\ldots,e_{k}\\
W,[d_0,e_0],\ldots,[d_{k},e_{k}]\text{ coprime}}}\frac{\lambda_{d_0,\ldots,d_{k}}(F_1)\lambda_{e_0,\ldots,e_{k}}(F_2)}{[d_0,e_0]\cdots[d_{k},e_{k}]}\notag\\
=&\frac{1+o_w(1)}{(\log R)^{{k+1}}}\cdot\frac{W^{{k}+1}}{\phi(W)^{{k}+1}}\int_{\R^{{k}+1}}\frac{\partial^{{k}+1}F_1(t_0,\ldots,t_{k})}{\partial t_0\cdots\partial t_{k}}\cdot\frac{\partial^{{k}+1}F_2(t_0,\ldots,t_{k})}{\partial t_0\cdots\partial t_{k}}d t_0\cdots d t_{k}.
\end{align}
And (\ref{maynardid}) is also valid if the denominator $[d_0,e_0]\cdots[d_{k},e_{k}]$ in the left side is replaced by $\phi([d_0,e_0]\cdots[d_{k},e_{k}])$.
\end{lemma}\begin{proof}
See \cite[Proposition 5]{T} or \cite[Lemma 30]{Pol14}.
\end{proof}
In this section, we shall establish an analogue of Maynard's sieve method for the exponential sums over primes.
Let $o_w(1)$ denote a quantity which tends to $0$ as $w\to+\infty$.
\begin{proposition}\label{pinhiaqthetaT} Let $\epsilon>0$ be a constant.
Suppose that $1\leq a\leq q\leq N^{\frac13-\epsilon}R^{-2}$ with $(a,q)=1$
and $|\theta|\leq N^{\epsilon-1}$.
If $q\mid W$, then for any $0\leq i\leq k$,
\begin{align}\label{pinhiaqtheta}
&\sum_{\substack{n\sim N\\ 
n\equiv b\pmod{W}}}\varpi(n+h_i)e\bigg((n+h_i)\bigg(\frac aq+\theta\bigg)\bigg)\cdot\Omega_n(F)\notag\\
=&e\bigg(\frac{a(b+h_i)}{q}\bigg)\cdot\frac{1+o_w(1)}{(\log R)^{{k}}}\cdot\frac{W^{{k}}}{\phi(W)^{{k+1}}}\cdot\cJ_i(F)\sum_{n\sim N}e(n\theta),
\end{align}
where
$$
\cJ_i(F)=\int_{\R^{{k}}}\bigg(\frac{\partial^{{k}}F(t_0,\ldots,t_{i-1},0,t_{i+1},\ldots,t_{k})}{\partial t_0\cdots\partial t_{i-1}\partial t_{i+1}\cdots\partial t_{k}}\bigg)^2d t_0\cdots d t_{i-1}d t_{i+1}\cdots d t_{k}.
$$
Otherwise,
\begin{equation}\label{pinhiaqtheta2}
\sum_{\substack{n\sim N\\ 
n\equiv b\pmod{W}}}\varpi(n+h_i)e\bigg((n+h_i)\bigg(\frac aq+\theta\bigg)\bigg)\cdot\Omega_n\ll_\epsilon\frac{1}{w^{1-\epsilon}}\cdot \frac{N}{(\log R)^{{k}}}\cdot\frac{W^{{k}}}{\phi(W)^{{k+1}}}.
\end{equation}
\end{proposition}
Below we shall fix $F(t_0,\ldots,t_k)$ as a smooth function whose support lies on (\ref{area}). For convenience, abbreviate $\lambda_{d_0,\ldots,d_k}(F)$, $\Omega_{n}(F)$ and $\cJ_i(F)$ as $\lambda_{d_0,\ldots,d_k}$, $\Omega_{n}$ and $\cJ_i$ respectively.

The following lemma is an analogue of the Bombieri-Vinogradov theorem for the exponential sums over primes, which was proved by Liu and Zhan \cite[Theorem 3]{LZ97}.
\begin{lemma}\label{BVexpsum}
For any $A>0$, there exists $B=B(A)>0$ such that
$$
\sum_{q\leq K}\max_{(a,q)=1}\max_{|\theta|\leq\delta}\bigg|
\sum_{n\sim x}\varpi(n)e\bigg(n\bigg(\frac aq+\theta\bigg)\bigg)-\frac{\mu(q)}{\phi(q)}\sum_{n\sim x}e(n\theta)\bigg|\ll \frac{x}{\log^A x},
$$
where $1\leq K\leq x^{\frac13}\log^{-B}x$ and $\delta=K^{-3}\log^{-B}x$.
\end{lemma}
Define
\begin{equation}\label{Rqdelta}
\cR_{q,\delta}(x)=\max_{\substack{1\leq a\leq q\\ (a,q)=1}}\max_{|\theta|\leq\delta}\bigg|
\sum_{n\sim x}\varpi(n)e\bigg(n\bigg(\frac aq+\theta\bigg)\bigg)-\frac{\mu(q)}{\phi(q)}\sum_{n\sim x}e(n\theta)\bigg|.
\end{equation}

For an assertion $P$, define
$$
\1_P=\begin{cases}1,&\text{the assertion }P\text{ holds},\\
0,&\text{otherwise}.
\end{cases}
$$
For example, $\1_A(x)=\1_{x\in A}$. Let
$$
\cZ_q=\{d:\,(d,q)\text{ and }q/(d,q)\text{ are co-prime}\}.
$$
\begin{lemma} \label{piaqthetaL}
Suppose that $1\leq a\leq q$ and $(a,q)]=1$. Let $D$ be a positive integer with $(b,D)=1$. 
Then
\begin{align}\label{piaqtheta}
&\sum_{\substack{n\sim x\\ n\equiv b\pmod{D}}}\varpi(n)e\bigg(n\bigg(\frac aq+\theta\bigg)\bigg)\notag\\
=&\mu\bigg(\frac{q}{(D,q)}\bigg)e\bigg(\frac{ab\cdot\bar{v}_{D,q}}{(D,q)}\bigg)\cdot\frac{\1_{D\in\cZ_q}}{\phi([D,q])}\sum_{\substack{n\sim x}}e(n\theta)+O\bigg(\sum_{t\mid [D,q]}\frac{(D,t)}D\cdot\cR_{t,\delta}(x)\bigg)
\end{align}
for any $\theta$ with $|\theta|\leq\delta$, where $\bar{v}_{D,q}$ is an integer such that
$$
\bar{v}_{D,q}\cdot\frac{q}{(D,q)}\equiv1\pmod{(D,q)}
$$
if $D\in\cZ_q$, and $\bar{v}_{D,q}=0$ otherwise.
\end{lemma}
\begin{proof} Clearly
$$
\sum_{\substack{n\sim x\\ n\equiv b\pmod{D}}}\varpi(n)e\bigg(n\bigg(\frac aq+\theta\bigg)\bigg)=\frac1D\sum_{r=1}^De\bigg(-\frac{br}{D}\bigg)\sum_{\substack{n\sim x}}\varpi(n)e\bigg(n\bigg(\frac rD+\frac aq+\theta\bigg)\bigg).
$$
Write
$$
\frac{r}{D}+\frac aq=\frac{s_{r}}{t_{r}},
$$
where $t_{r}$ is a common multiple of $D$ and $q$ and $(s_{r},t_{r})=1$. 
Then in view of (\ref{Rqdelta}),
\begin{align}\label{bDaq1}
&\sum_{\substack{n\sim x\\ n\equiv b\pmod{D}}}\varpi(n)e\bigg(n\bigg(\frac aq+\theta\bigg)\bigg)\notag\\
=&
\frac1D\sum_{r=1}^D\frac{\mu(t_r)}{\phi(t_r)}\cdot e\bigg(-\frac{br}{D}\bigg)\sum_{n\sim x}e(n\theta)+O\bigg(\sum_{t\mid [D,q]}\frac{M_t\cR_{t,\delta}(x)}D\bigg),
\end{align}
where
$$
M_t=|\{1\leq r\leq D:\,t_r=t\}|.
$$
On the other hand, by the Siegel-Walﬁsz theorem, for any $A>0$, we have
\begin{align}\label{bDaq2}
\sum_{\substack{n\sim x\\ n\equiv b\pmod{D}}}\varpi(n)e\bigg(n\cdot\frac aq\bigg)=
&\sum_{\substack{1\leq c\leq q\\ (c,q)=1}}
e\bigg(\frac{ac}q\bigg)\sum_{\substack{n\sim x\\ n\equiv b\pmod{D}\\
n\equiv c\pmod{q}}}\varpi(n)\notag\\
=&\frac{x}{\phi([D,q])}\sum_{\substack{1\leq c\leq q,\ (c,q)=1\\ c\equiv b\pmod{(D,q)}}}
e\bigg(\frac{ac}q\bigg)+O\bigg(\frac{qx}{(\log x)^A}\bigg),
\end{align}
whenever $x\geq e^{[D,q]}$.
Write $u=(D,q)$ and $v=q/u$. We have
\begin{align*}
\sum_{\substack{1\leq c\leq q,\ (c,q)=1\\ c\equiv b\pmod{(D,q)}}}
e\bigg(\frac{ac}q\bigg)=&\sum_{d\mid q}\mu(d)\sum_{\substack{1\leq c\leq q,\ d\mid c\\ c\equiv b\pmod{u}}}e\bigg(\frac{ac}q\bigg)\\
=&\sum_{d\mid q}\mu(d)\sum_{\substack{1\leq t\leq v\\ 
ut+b\equiv 0\pmod{d}}}e\bigg(\frac{a(ut+b)}q\bigg).
\end{align*}
Since $(D,b)=1$ and $u\mid D$, $ut+b\equiv0\pmod{d}$ for some integer $t$ if and only if $(u,d)=1$ and $d\mid v$. 
It is easy to see that
$$
\sum_{\substack{1\leq t\leq v\\ 
ut+b\equiv 0\pmod{d}}}e\bigg(\frac{a(ut+b)}q\bigg)=0
$$
unless $(u,v)=1$ and $d=v$. Assume that $(u,v)=1$. Let $\bar{v}_u$ be an integer with $\bar{v}_uv\equiv1\pmod{u}$. When $ut+b\equiv0\pmod{v}$,
$$
\frac{ut+b}{v}\equiv(ut+b)\bar{v}_u\equiv b\bar{v}_u\pmod{u}.
$$
Thus we have
\begin{align*}
\sum_{\substack{1\leq c\leq q,\ (c,q)=1\\ c\equiv b\pmod{(D,q)}}}
e\bigg(\frac{ac}q\bigg)
=\begin{cases}\mu(v)e(ab\bar{v}_u/u),&\text{if }(u,v)=1,\\
0,&\text{otherwise}.\end{cases}
\end{align*}
Fix $q$ and $D$, and let $x\geq e^{[D,q]}$. By Lemma \ref{BVexpsum}, we have $\cR_{t,\delta}(x)\ll x(\log x)^{-A}$ for any $t\leq [D,q]$ provided that $\delta$ is sufficiently small.
Thus setting $\theta=0$ in (\ref{bDaq1}) and comparing (\ref{bDaq1}) and (\ref{bDaq2}), we get
$$
\frac1D\sum_{r=1}^D\frac{\mu(t_r)}{\phi(t_r)}\cdot e\bigg(-\frac{br}{D}\bigg)=\frac{\mu(v)e\big(\frac{ab\bar{v}_u}{u}\big)\cdot\1_{(u,v)=1}}{\phi([D,q])}.
$$

Suppose that $t$ is a divisor of $[D,q]$ and $t_r=t$. Note that
$$
\frac{r}{D}+\frac{a}{q}=\frac{rq+a D}{Dq}=\frac{s_r}{t}.
$$
Thus $r$ must satisfy the congruence
$$
rq+aD\equiv0\pmod{Dq/t}.
$$
So
$$
M_t\leq\frac{(Dq/t,q)}{Dq/t}\cdot D=(D,t).
$$
(\ref{piaqtheta}) is concluded.
\end{proof}

We are ready to give the proof of Proposition \ref{pinhiaqthetaT}. Clearly we only need to consider the case $i=0$.
Let $S_W=\{d\in\N:\,(d,W)=1\}$.
By Lemma \ref{piaqthetaL}, we have
\begin{align*}
&\sum_{\substack{n\sim N\\ 
n\equiv b\pmod{W}}}\varpi(n+h_0)e\bigg((n+h_0)\bigg(\frac aq+\theta\bigg)\bigg)\cdot\bigg(\sum_{d_i\mid n+h_i}\lambda_{d_0,\ldots,d_{k}}\bigg)^2\\
=&\sum_{\substack{d_1,\ldots,d_{k},e_1,\ldots,e_{k}\in S_W\\
[d_1,e_1],\ldots,[d_{k},e_{k}]\text{ coprime}}}\lambda_{1,d_1,\ldots,d_{k}}\lambda_{1,e_1,\ldots,e_{k}}
\sum_{\substack{N+h_0\leq n\leq 2N+h_0\\ 
n\equiv b+h_0\pmod{W}\\
n\equiv h_0-h_i\pmod{[d_i,e_i]}}}\varpi(n)e\bigg(n\bigg(\frac aq+\theta\bigg)\bigg)\\
=&\sum_{n\sim N}e(n\theta)\cdot\sum_{\substack{d_1,\ldots,d_{k},e_1,\ldots,e_{k}\in S_W\\
[d_1,e_1],\ldots,[d_{k},e_{k}]\text{ coprime}}}
\frac{\lambda_{1,d_1,\ldots,d_{k}}\lambda_{1,e_1,\ldots,e_{k}}\cdot\Delta_{d_1,e_1,\ldots,d_{k},e_{k}}\cdot\1_{[W,d_1,e_1,\ldots,d_{k},e_{k}]\in\cZ_q}}{\phi([W,d_1,e_1,\ldots,d_{k},e_{k},q])}\\
&+O\bigg(\sum_{\substack{d_1,\ldots,d_{k},e_1,\ldots,e_{k}\in S_W\\
d_1\cdots d_{k},e_1\cdots e_{k}\leq R\\
[d_1,e_1],\ldots,[d_{k},e_{k}]\text{ coprime}}}\sum_{t\mid [W,d_1,e_1,\ldots,d_{k},e_{k},q]}\frac{([W,d_1,e_1,\ldots,d_{k},e_{k}],t)}{[W,d_1,e_1,\ldots,d_{k},e_{k}]}\cdot \cR_{t,\delta}(N)\bigg),
\end{align*}
where
$$
\Delta_{d_1,e_1,\ldots,d_{k},e_{k}}=
\mu\bigg(\frac{q}{([W,d_1,e_1,\ldots,d_{k},e_{k}],q)}\bigg)\cdot e\bigg(\frac{a(b+h_0)\cdot\bar{v}_{[W,d_1,e_1,\ldots,d_{k},e_{k}],q}}{([W,d_1,e_1,\ldots,d_{k},e_{k}],q)}\bigg).
$$

We firstly consider the remainder term.
Assume that $t\leq qWR^2$. We have
\begin{align*}
\sum_{\substack{1\leq D\leq N}}\tau(D)^{2k}\cdot\frac{(D,t)}{D}\leq
\sum_{s\mid t}\tau(s)^{2k}\sum_{\substack{1\leq d\leq N/s}}\frac{\tau(d)^{2k}}{d}\ll
\tau(t)^{2k+1}\cdot(\log N)^{4^{k}},
\end{align*}
where we use the known result
$$
\sum_{1\leq d\leq x}\tau(d)^k\ll x(\log x)^{2^k-1}.
$$
It follows from Lemma \ref{BVexpsum} with $K=N^{\frac13-\frac\epsilon{2}}$ that
\begin{align*}
&\sum_{\substack{d_1,\ldots,d_{k},e_1,\ldots,e_{k}\in S_W\\
d_1\cdots d_{k},e_1\cdots e_{k}\leq R\\
[d_1,e_1],\ldots,[d_{k},e_{k}]\text{ coprime}}}\sum_{t\mid [W,d_1,e_1,\ldots,d_{k},e_{k},q]}\frac{([W,d_1,e_1,\ldots,d_{k},e_{k}],t)}{[W,d_1,e_1,\ldots,d_{k},e_{k}]}\cdot \cR_{t,\delta}(N)\\
\ll&\sum_{t\leq N^{\frac13-\frac\epsilon{2}}}\cR_{t,\delta}(N)\sum_{\substack{1\leq D\leq N}}\tau(D)^{2k}\cdot\frac{(D,t)}{D}
\\
\ll&(\log N)^{4^{k}}\sum_{t\leq N^{\frac13-\frac\epsilon{2}}}\tau(t)^{2k+1}\cR_{t,\delta}(N)
\ll\frac{N}{(\log N)^{k+2}}.
\end{align*}

Let us turn to the main term. Evidently the main term will vanish if $W\not\in\cZ_q$. Below we assume that $W\in\cZ_q$.
Let $\fX_q$ be the set of all those $(\fd_0,\fe_0,\ldots,\fd_{k},\fe_{k})$  satisfying that

\medskip(i) $\fd_0,\fe_0,\ldots,\fd_{k},\fe_{k}$ are the divisors of $q$;

\medskip(ii) $\fd_0,\fe_0,\ldots,\fd_{k},\fe_{k}$ are square-free;

\medskip(iii) $\fd_0,\fe_0,\ldots,\fd_{k},\fe_{k}\in\cZ_q$;

\medskip(iv) $W,[\fd_0,\fe_0],\ldots,[\fd_{k},\fe_{k}]$ are co-prime.\medskip

Suppose that $(\fd_0,\fe_0,\ldots,\fd_{k},\fe_{k})\in\fX_q$.
\begin{align*}
&\sum_{\substack{d_1,\ldots,d_{k},e_1,\ldots,e_{k}\in S_W\\
[d_1,e_1],\ldots,[d_{k},e_{k}]\text{ coprime}\\
d_1,\ldots,d_{k},e_1,\ldots,e_{k}\in\cZ_q\\
(d_i,q)=\fd_i,\ (e_i,q)=\fe_i
}}\frac{\lambda_{1,d_1,\ldots,d_{k}}\lambda_{1,e_1,\ldots,e_{k}}}{\phi([d_0,e_0]\cdots[d_{k},e_{k}])}\\
=&\sum_{\substack{(d_1^*,\ldots,d_{k}^*,e_1^*,\ldots,e_{k}^*)\in S_{[W,q]}\\
[d_1^*,e_1^*],\ldots,[d_{k}^*,e_{k}^*]\text{ coprime}\\}}\frac{\lambda_{1,\fd_1d_1^*,\ldots,\fd_{k}d_{k}^*}\lambda_{1,\fe_1e_1^*,\ldots,\fe_{k}e_{k}^*}}{\phi([\fd_1,\fe_1][d_1^*,e_1^*]\cdots[\fd_{k},\fe_{k}][d_{k}^*,e_{k}^*])}\\
=&\prod_{j=1}^{k}\frac{\mu(\fd_j)\mu(\fe_j)}{\phi([\fd_j,\fe_j])}\sum_{\substack{(d_1^*,\ldots,d_{k}^*,e_1^*,\ldots,e_{k}^*)\in S_{[W,q]}\\
W,[d_1^*,e_1^*],\ldots,[d_{k}^*,e_{k}^*]\text{ coprime}}}\frac{\lambda_{1,d_1^*,\ldots,d_{k}^*}(F_{\fd_1,\ldots,\fd_{k}})\lambda_{1,e_1^*,\ldots,e_{k}^*}(F_{\fe_1,\ldots,\fe_{k}})}{\phi([d_1^*,e_1^*]\cdots[d_{k}^*,e_{k}^*])},
\end{align*}
where
$$
F_{\fd_1,\ldots,\fd_{k}}(t_1,\ldots,t_{k})=F\bigg(0,t_1+\frac{\log\fd_1}{\log R},\ldots,t_{k}+\frac{\log\fd_{k}}{\log R}\bigg).
$$
According to Lemma \ref{maynard},
\begin{align*}
&\sum_{\substack{(d_1^*,\ldots,d_{k}^*,e_1^*,\ldots,e_{k}^*)\in S_{Wq}\\
W,[d_1^*,e_1^*],\ldots,[d_{k}^*,e_{k}^*]\text{ are coprime}}}\frac{\lambda_{1,d_1^*,\ldots,d_{k}^*}(F_{\fd_1,\ldots,\fd_{k}})\lambda_{1,e_1^*,\ldots,e_{k}^*}(F_{\fe_1,\ldots,\fe_{k}})}{\phi([d_1^*,e_1^*]\cdots[d_{k}^*,e_{k}^*])}\\
=&\frac{1+o(1)}{(\log R)^{{k}}}\cdot\frac{[W,q]^{{k}}}{\phi([W,q])^{{k}}}\int_{\R^{{k}}}\frac{\partial^{{k}}F_{\fd_1,\ldots,\fd_{k}}}{\partial t_1\cdots\partial t_{k}}\cdot\frac{\partial^{{k}}F_{\fe_1,\ldots,\fe_{k}}}{\partial t_1\cdots\partial t_{k}}d t_1\cdots d t_{k}.
\end{align*}

Suppose that $q\nmid W$. 
Letting
$$
\Theta_F=\max_{0\leq t_1,\ldots,t_{k}\leq 1}\bigg|\frac{\partial^{{k}}F(0,t_1,\ldots,t_{k})}{\partial t_1\cdots\partial t_{k}}\bigg|,
$$
we get
\begin{align*}
\bigg|\sum_{\substack{d_1,\ldots,d_{k},e_1,\ldots,e_{k}\in S_W\\
[d_1,e_1],\ldots,[d_{k},e_{k}]\text{ coprime}\\
d_1,\ldots,d_{k},e_1,\ldots,e_{k}\in\cZ_q\\
(d_i,q)=\fd_i,\ (e_i,q)=\fe_i
}}\frac{\lambda_{1,d_1,\ldots,d_{k}}\lambda_{1,e_1,\ldots,e_{k}}}{\phi([d_0,e_0]\cdots[d_{k},e_{k}])}\bigg|\leq\prod_{j=1}^{k}\frac{1}{\phi([\fd_j,\fe_j])}\cdot\frac{1+o_w(1)}{(\log R)^{{k}}}\cdot\frac{[W,q]^{{k}}}{\phi([W,q])^{{k}}}\cdot \Theta_F^2.
\end{align*}
Thus
\begin{align*}
&\sum_{\substack{d_1,\ldots,d_{k},e_1,\ldots,e_{k}\in S_W\\
d_1,e_1,\ldots,d_{k},e_{k}\in\cZ_q\\
[d_1,e_1],\ldots,[d_{k},e_{k}]\text{ coprime}}}
\Delta_{d_1,e_1,\ldots,d_{k},e_{k}}
\cdot\frac{\lambda_{1,d_1,\ldots,d_{k}}\lambda_{1,e_1,\ldots,e_{k}}}{\phi([W,d_1,e_1,\ldots,d_{k},e_{k},q])}\\
=&\sum_{\substack{\fd_1,\fe_1,\ldots,\fd_{k},\fe_{k}\in\fX_q}}\frac{\Delta_{d_1,e_1,\ldots,d_{k},e_{k}}}{\phi\big(\frac{[W,q]}{[\fd_1,\fe_1]\cdots[\fd_{k},\fe_{k}]}\big)}
\sum_{\substack{d_1,\ldots,d_{k},e_1,\ldots,e_{k}\in S_W\\
[d_1,e_1],\ldots,[d_{k},e_{k}]\text{ are coprime}\\
d_1,\ldots,d_{k},e_1,\ldots,e_{k}\in\cZ_q\\
(d_i,q)=\fd_i,\ (e_i,q)=\fe_i
}}\frac{\lambda_{1,d_1,\ldots,d_{k}}\lambda_{1,e_1,\ldots,e_{k}}}{\phi([d_1,e_1]\cdots[d_{k},e_{k}])}\\
\ll&\sum_{\substack{\fd_1,\fe_1,\ldots,\fd_{k},\fe_{k}\in\fX_q}}\frac{1}{\phi\big(\frac{[W,q]}{[\fd_1,\fe_1]\cdots[\fd_{k},\fe_{k}]}\big)}
\cdot\prod_{j=1}^{k}\frac{1}{\phi([\fd_j,\fe_j])}\cdot\frac{1}{(\log R)^{{k}}}\cdot\frac{[W,q]^{{k}}}{\phi([W,q])^{{k}}}\cdot \Theta_F^2\\
=&\frac{\Theta_F^2}{(\log R)^{{k}}}\cdot\frac{[W,q]^{k}}{\phi([W,q])^{k+1}}\sum_{\substack{\fd_1,\fe_1,\ldots,\fd_{k},\fe_{k}\in\fX_q}}1.
\end{align*}

Let $q_*=q/(W,q)$. Since $q\nmid W$ and $W\in\cZ_q$, $q_*$ must have at least one prime factor greater than $w$, i.e., $q_*>w$. And $\fd_1,\fe_1,\ldots,\fd_{k},\fe_{k}\in\fX_q$ implies that $\fd_1,\fe_1,\ldots,\fd_{k},\fe_{k}$ all divide $q_*$. Hence
\begin{align*}
\frac{[W,q]^{k}}{\phi([W,q])^{k+1}}\sum_{\substack{\fd_1,\fe_1,\ldots,\fd_{k},\fe_{k}\in\fX_q}}1
\leq\frac{W^{k}}{\phi(W)^{k+1}}\cdot\frac{\tau(q_*)^{2k}q_*^{k}}{\phi(q_*)^{k+1}}\ll_\epsilon
\frac{W^{k}}{\phi(W)^{k+1}}\cdot\frac{1}{w^{1-\epsilon}}.
\end{align*}

Finally, if $q\mid W$, then evidently $\fX_q=\{(1,1,\ldots,1)\}$. So
\begin{align*}
&\sum_{\substack{d_1,\ldots,d_{k},e_1,\ldots,e_{k}\in S_W\\
d_1,e_1,\ldots,d_{k},e_{k}\in\cZ_q\\
[d_1,e_1],\ldots,[d_{k},e_{k}]\text{ coprime}}}
\Delta_{d_1,e_1,\ldots,d_{k},e_{k}}
\cdot\frac{\lambda_{1,d_1,\ldots,d_{k}}\lambda_{1,e_1,\ldots,e_{k}}}{\phi([W,d_1,e_1,\ldots,d_{k},e_{k},q])}\\
=&e\bigg(\frac{a(b+h_0)}{q}\bigg)\cdot\frac{1}{\phi(W)}
\sum_{\substack{d_1,\ldots,d_{k},e_1,\ldots,e_{k}\in S_W\\
[d_1,e_1],\ldots,[d_{k},e_{k}]\text{ are coprime}
}}\frac{\lambda_{1,d_1,\ldots,d_{k}}\lambda_{1,e_1,\ldots,e_{k}}}{\phi([d_1,e_1]\cdots[d_{k},e_{k}])}\\
=&e\bigg(\frac{a(b+h_0)}{q}\bigg)\cdot\frac{1+o_w(1)}{(\log R)^{{k}}}\cdot\frac{W^{{k}}}{\phi(W)^{{k+1}}}\int_{\R^{{k}}}\bigg(\frac{\partial^{{k}}F(0,t_1,\ldots,t_{k})}{\partial t_1\cdots\partial t_{k}}\bigg)^2d t_1\cdots d t_{k}.
\end{align*}
\qed
\begin{remark}
If we set $a=q=1$ and $\theta=0$ in (\ref{pinhiaqtheta}), we can get 
\begin{align}\label{pinhisum}
\sum_{\substack{n\sim N\\ 
n\equiv b\pmod{W}}}\varpi(n+h_i)\cdot\Omega_n
=(1+o_w(1))\cdot\frac{\cJ_iN}{(\log R)^{{k}}}\cdot\frac{W^{{k}}}{\phi(W)^{{k+1}}},
\end{align}
which is one of two key formulas used in the proof of the Maynard-Tao theorem.
And the other one is
\begin{align}\label{nhisum}
\sum_{\substack{n\sim N\\ 
n\equiv b\pmod{W}}}\Omega_n
=(1+o_w(1))\cJ_*\cdot\frac{N}{(\log R)^{{k+1}}}\cdot\frac{W^{{k}}}{\phi(W)^{{k+1}}},
\end{align}
where
$$
\cJ_*=\int_{\R^{{k}}}\bigg(\frac{\partial^{{k+1}}F(t_0,\ldots,t_{k})}{\partial t_0\cdots\partial t_{k}}\bigg)^2d t_0\cdots d t_{k}.
$$
\end{remark}

\section{Reducing to the Kronecker system}
\setcounter{lemma}{0}\setcounter{theorem}{0}\setcounter{proposition}{0}\setcounter{corollary}{0}\setcounter{remark}{0}
\setcounter{equation}{0}

Suppose that $(\cX,\sB_\cX,\mu,T)$ is a measure-preserving system and $T$ is invertible.
For any $f(x),g(x)\in L^2(\cX,\sB_\cX,\mu)$, let
$$
\langle f,g\rangle:=\int_\cX f\cdot\overline{g}d\mu 
$$
denote the inner product over $L^2(\cX,\sB_\cX,\mu)$. In particular, let $\|f\|_2:=\sqrt{\langle f,f\rangle}$ denote the $L^2$-norm of $f$.

Let $\sK\subseteq \sB_\cX$ be the smallest sub-$\sigma$-algebra with respect to which all eigenfunctions of $T$ are measurable.
For any $f(x)\in L_\nu^2$, let $\E(f|\sK)$ denote the conditional expectation of $f$ with respect to $\sK$
(cf. \cite[Theorem 5.1]{EW11}).
It is known that 
$$
\int_A fd\mu=\int_A \E(f|\sK)d\mu
$$
for any $A\in\sK$. And if $f$ is non-negative, then $\E(f|\sK)$  is also non-negative almost everywhere.
In fact, $\E(f|\sK)$ is the orthogonal projection of $f$ over the close Hilbert subspace $L^2(\cX,\sK,\mu)$.
Write $f_1=\E(f|\sK)$ and $f_2=f-\E(f|\sK)$. Then
$f_2$ is orthogonal to $L^2(\cX,\sK,\mu)$.
Since $T^{-1}(\sK)\subseteq\sK$, $T^nf_1\in L^2(\cX,\sK,\mu)$ for any $n\in\N$. Thus we have
$$
\langle f,T^nf\rangle=\langle f_1,T^nf_1\rangle+\langle f_2,T^nf_2\rangle.
$$

In this section, we shall focus on those $f(x)$ with $\E(f|\sK)=0$.
\begin{proposition}\label{fK0}
Suppose that $\E(f|\sK)=0$. Then for each $0\leq i\leq k$,
\begin{equation}
\sum_{\substack{n\sim N\\ 
n\equiv b\pmod{W}}}\varpi(n+h_i)\Omega_n\cdot\langle f,T^{n+h_i-1}f\rangle=o_w\bigg(\frac{N}{(\log R)^{k}}\cdot\frac{W^{k}}{\phi(W)^{k+1}}\bigg),
\end{equation}
whenever $N$ is sufficiently large with respect to $w$.
\end{proposition}
Clearly we only need to consider the case $i=0$.
Let $\T:=\R/\Z$ denote the $1$-dimensional torus. For convenience, let $|\cdot|_\T$ denote the norm over $\T$, i.e., for any $x\in\R$
$$
|x|_\T:=\min\{|x-n|:\,n\in\Z\}.
$$
Let $P=N^{\frac13-\frac1{99}}$ and $Q=N^{1-\frac1{49}}$.
Define
$$
\fM_N^{(a,q)}=\{\alpha\in\T:\,|\alpha-a/q|_\T\leq q^{-1}Q^{-1}\}.
$$
By a well-known result of Dirichlet, we have
$$
\T=\bigcup_{\substack{1\leq a\leq q\leq Q\\ (a,q)=1}}\fM_N^{(a,q)}.
$$
Let
$$
\fM_N=\bigcup_{\substack{1\leq a\leq q\leq P\\ (a,q)=1}}\fM_N^{(a,q)}
$$
and let $\fm_N=\T\setminus\fM_N$. 

By the well-known  Herglotz theorem (cf. \cite[Chapter 1.8]{He83}), there exists a non-negative measure $\upsilon$ on $\T$ such that
$$
\langle f,T^nf\rangle=\int_0^1 e(n\alpha)d\upsilon(\alpha).
$$
First, we consider the integrals on the minor arc $\fm_N$. The following lemma is due to Balog and Perelli \cite{BP85}.
\begin{lemma}\label{minorL}
Suppose that $(a,q)=(b,D)=1$.
Then letting $u_D=(D,q)$, we have
\begin{equation}
\sum_{\substack{n\leq x\\ n\equiv b\pmod{D}}}\varpi(n)e\bigg(n\cdot\frac aq\bigg)\ll
(\log x)^3\bigg(\frac{u_Dx}{Dq^{\frac 12}}+\frac{q^{\frac12}x^{\frac12}}{u_D^{\frac12}}+\frac{x^{\frac45}}{D^{\frac25}}\bigg).
\end{equation}
\end{lemma}
Applying Lemma \ref{minorL} and the partial summation, we can get
\begin{align*}
\sum_{\substack{n\sim N\\ n\equiv b\pmod{D}}}\varpi(n)e\bigg(n\bigg(\frac aq+\theta\bigg)\bigg)\ll(1+\theta N)(\log N)^3\cdot\bigg(\frac{u_DN}{Dq^{\frac 12}}+\frac{q^{\frac12}N^{\frac12}}{u_D^{\frac12}}+\frac{N^{\frac45}}{D^{\frac25}}\bigg),
\end{align*}
where $u_D=(D,q)$.
It follows that
\begin{align*}
&\sum_{\substack{n\sim N\\ 
n\equiv b\pmod{W}}}\varpi(n+h_0)e\bigg((n+h_0)\bigg(\frac aq+\theta\bigg)\bigg)\cdot\Omega_n\\
\ll&\sum_{\substack{d_1,\ldots,d_{k},e_1,\ldots,e_{k}\in S_W\\
d_1\cdots d_{k},e_1\cdots e_{k}\leq R
\\
[d_1,e_1],\ldots,[d_{k},e_{k}]\text{ coprime}}}\bigg|\sum_{\substack{n\sim N\\ 
n\equiv b+h_0\pmod{W}\\
n\equiv h_0-h_i\pmod{[d_i,e_i]}}}\varpi(n)e\bigg(n\bigg(\frac aq+\theta\bigg)\bigg)\bigg|\\
\ll&(1+\theta N)(\log N)^3\sum_{D\leq WR^2}\tau(D)^{2k}\cdot \bigg(\frac{u_DN}{Dq^{\frac 12}}+\frac{q^{\frac12}N^{\frac12}}{u_D^{\frac12}}+\frac{N^{\frac45}}{D^{\frac25}}\bigg).
\end{align*}
Note that
\begin{align*}
\sum_{D\leq WR^2}\tau(D)^{2k}\cdot\frac{u_D}{D}\leq&
\sum_{u\mid q}\sum_{\substack{D\leq WR^2\\ u\mid D}}\frac{\tau(D)^{2k}}{D/u}
\leq\sum_{u\mid q}\tau(u)^{2k}\sum_{\substack{v\leq WR^2/u}}\frac{\tau(v)^{2k}}{v}\\
\ll&
(\log N)^{4^k}\sum_{u\mid q}\tau(u)^{2k}\ll
(\log N)^{4^k}\cdot\tau(q)^{2k+1}.
\end{align*}
Hence for any $\epsilon>0$,
\begin{align}\label{minor}
&\sum_{\substack{n\sim N\\ 
n\equiv b\pmod{W}}}\varpi(n+h_0)e\bigg((n+h_0)\bigg(\frac aq+\theta\bigg)\bigg)\cdot\Omega_n\notag\\
\ll_\epsilon&\frac{N^{1+\epsilon}}{q^{\frac 12-\epsilon}}+\frac{\theta N^{2+\epsilon}}{q^{\frac 12-\epsilon}}+q^{\frac12} N^{\frac12+\epsilon}R^2+\theta q^{\frac12} N^{\frac32+\epsilon}R^2+N^{\frac45+\epsilon} R^{\frac65}+\theta N^{\frac95+\epsilon} R^{\frac65}.
\end{align}

Suppose that $\alpha\in\fm_N$. Then $\alpha=a/q+\theta$ where $P<q\leq Q$, $(a,q)=1$ and $|\theta|\leq q^{-1}Q^{-1}$.
By (\ref{minor}),
\begin{align}\label{minorN}
&\sum_{\substack{n\sim N\\ 
n\equiv b\pmod{W}}}\varpi(n+h_0)e\big((n+h_0)\alpha\big)\cdot\Omega_n\notag\\
\ll_\epsilon&\frac{N^{1+\frac1{99}}}{q^{\frac 12-\frac1{99}}}+\frac{N^{2+\frac1{99}}}{q^{\frac 32-\frac1{99}}Q}+q^{\frac12} N^{\frac12+\frac1{99}}R^2+\frac{N^{\frac32+\frac1{99}}R^2}{q^{\frac12}Q}+N^{\frac45+\frac1{99}} R^{\frac65}+\frac{N^{\frac95+\frac1{99}} R^{\frac65}}{qQ}\notag\\
\leq&\frac{N^{1+\frac1{99}}}{P^{\frac 12-\frac1{99}}}+\frac{N^{2+\frac1{99}}}{P^{\frac 32-\frac1{99}}Q}+Q^{\frac12} N^{\frac12+\frac1{99}}R^2+\frac{N^{\frac32+\frac1{99}}R^2}{P^{\frac12}Q}+N^{\frac45+\frac1{99}} R^{\frac65}+\frac{N^{\frac95+\frac1{99}} R^{\frac65}}{PQ}\notag\\
\ll&N^{1-\frac1{999}}.
\end{align}
Hence
$$
\sum_{\substack{n\sim N\\ 
n\equiv b\pmod{W}}}\varpi(n+h_0)\Omega_n\int_{\fm_N} e\big((n+h_0-1)\alpha\big)d\upsilon(\alpha)\ll N^{1-\frac1{999}}.
$$

Next, let us turn to the integrals on $\fM_N$.
Since $\E(f|\sK)=0$, we must have $f(x)$ is orthogonal to any 
eigenfunction of $T$.
It follows that $\upsilon$ is non-atomic, i.e., for any $\theta\in\T$,
$$
\lim_{\epsilon\to 0}\int_{\theta-\epsilon}^{\theta+\epsilon}1d\upsilon=0.
$$
For any $k\geq 1$, Let $X_k$ be the least positive integer such that if $N\geq X_k$, then
$$
\sum_{\substack{1\leq a\leq q\leq k\\ (a,q)=1}}\int_{\fM_N^{(a,q)}}1d\upsilon\leq\frac1k.
$$
Define
$\Psi(x)=k$ if $X_k\leq x<X_{k+1}$. Then $\Psi(x)$ tends to infinity as $x\to+\infty$, and
$$
\lim_{N\to+\infty}\sum_{\substack{1\leq a\leq q\leq \Psi(N)\\ (a,q)=1}}\int_{\fM_N^{(a,q)}}1d\upsilon
\leq\lim_{N\to+\infty}\frac1{\Psi(N)} =0.
$$

Assume that $N$ is sufficiently large such that
$$w\leq\frac13\log \Psi(N).$$ 
Then
$$
W=W_0\prod_{p\leq w}p\leq\Psi(N).
$$
Let
$$
\fM_N^*=\bigcup_{\substack{1\leq a\leq q\leq \Psi(N)\\ (a,q)=1}}\fM_N^{(a,q)}
$$
Then by (\ref{pinhiaqtheta}),
\begin{align*}
&\sum_{\substack{n\sim N\\ 
n\equiv b\pmod{W}}}\varpi(n+h_0)\Omega_n\int_{\fM_N^*}e\big((n+h_0-1)\alpha\big)d\upsilon(\alpha)
\\=&\sum_{\substack{1\leq a\leq q\leq \Psi(N)\\ (a,q)=1}}\int_{\fM_N^{(a,q)}}\bigg(\sum_{\substack{n\sim N\\ 
n\equiv b\pmod{W}}}\varpi(n+h_0)\Omega_n\cdot e\big((n+h_0-1)\alpha\big)\bigg)d\upsilon(\alpha)\\
\ll&\frac{N}{(\log R)^{k}}\cdot\frac{W^{k}}{\phi(W)^{k+1}}\sum_{\substack{1\leq a\leq q\leq \Psi(N)\\ (a,q)=1}}\int_{\fM_N^{(a,q)}}1d\upsilon=o\bigg(\frac{N}{(\log R)^{k}}\cdot\frac{W^{k}}{\phi(W)^{k+1}}\bigg).
\end{align*}
On the other hand, if $\fM_N^{(a,q)}\subseteq\fM_N\setminus\fM_N^*$, then we must have $q>\Psi(N)\geq W$, i.e., $q\nmid W$.
It follows from (\ref{pinhiaqtheta2}) that
\begin{align*}
&\sum_{\substack{n\sim N\\ 
n\equiv b\pmod{W}}}\varpi(n+h_0)\Omega_n\int_{\fM_N\setminus\fM_N^*}e\big((n+h_0-1)\alpha\big)d\upsilon(\alpha)
\\=&\sum_{\substack{\Psi(N)<q\leq P\\ 1\leq a\leq q,\ (a,q)=1}}\int_{\fM_N^{(a,q)}}\sum_{\substack{n\sim N\\ 
n\equiv b\pmod{W}}}\varpi(n+h_0)\Omega_n\cdot e\big((n+h_0-1)\alpha\big)d\upsilon(\alpha)\\
\ll&\frac{1}{\sqrt{w}}\cdot\frac{N}{(\log R)^{k}}\cdot\frac{W^{k}}{\phi(W)^{k+1}}\cdot\sum_{\substack{\Psi(N)<q\leq P\\ 1\leq a\leq q,\ (a,q)=1}}\int_{\fM_N^{(a,q)}}1d\upsilon\\
=&o_w\bigg(\frac{N}{(\log R)^{k}}\cdot\frac{W^{k}}{\phi(W)^{k+1}}\bigg).
\end{align*}
Thus
\begin{align*}
&\sum_{\substack{n\sim N\\ 
n\equiv b\pmod{W}}}\varpi(n+h_0)\Omega_n\cdot\langle f,T^{n+h_0-1}f\rangle\\
=&
\sum_{\substack{n\sim N\\ 
n\equiv b\pmod{W}}}\varpi(n+h_0)\Omega_n\int_{0}^1e\big((n+h_0-1)\alpha\big)d\upsilon(\alpha)\\
=&o_w\bigg(\frac{N}{(\log R)^{k}}\cdot\frac{W^{k}}{\phi(W)^{k+1}}\bigg).
\end{align*}
\qed

\section{Ergodic rotation on a torus}
\setcounter{lemma}{0}\setcounter{theorem}{0}\setcounter{proposition}{0}\setcounter{corollary}{0}\setcounter{remark}{0}
\setcounter{equation}{0}

In this section, we assume that $f(x)\in L^2(\cX,\sK,\mu)$. 
Let
$$
\|f\|_1:=\int_\cX|f|d\mu
$$
denote the $L^1$-norm of $f$.
\begin{lemma}
There exist a measure-preserving system $(\cG,\sB_\cG,\nu,S_\alpha)$ and  a map $\psi$ from $\cX$ to $\cG$ such that
\medskip

(1) $\cG$ is a compact abelian group and $\nu$ is the Haar measure on $\cG$; \medskip

(2) $S_\alpha:\,x\mapsto x+\alpha$ is an ergodic rotation on $\cG$, where $\alpha\in\cG$;\medskip

(3) $\psi\circ T=S_\alpha\circ\psi$ and $\nu=\mu\circ\psi^{-1}$;\medskip

(4) $\sK=\psi^{-1}(\sB_\cG)\pmod{\mu}$;\medskip
\end{lemma}
\begin{proof}
See \cite[Theorem 6.10]{EW11}.
\end{proof}
Here $(\cG,\sB_\cG,\nu,S_\alpha)$ is also called the {\it Kronecker factor } of $(\cX,\sB_\cX,\mu,T)$.
Since $\cG$ is a compact abelian group, $\cG$ is isomorphic to the direct sum of a finite abelian group and a torus.
So we may assume that $\cG=G\oplus\T^d$ and $\nu=\upsilon_G\times\fm_{\T^d}$, where $\upsilon_G$ is the discrete probability measure on $G$
and $\fm_{\T^d}$ is the Lebesgue measure on $\T^d$.
Below we may assume that the integer $W_0$ in (\ref{Ww}) is divisible by the cardinality of $G$.

\begin{proposition}\label{fK}
Suppose that $f(x)\in L^2(\cX,\sK,\mu)$ is a non-negative function with $0<\|f\|_2\leq 1$ and $0<\epsilon<\|f\|_1^2$. Then for each $0\leq i\leq k$
\begin{equation}
\sum_{\substack{n\sim N\\ 
n\equiv b\pmod{W}}}\varpi(n+h_i)\Omega_n\cdot\langle f,T^{n+h_i-1}f\rangle
\geq(\|f\|_1^2-\epsilon)\cdot\frac{\cJ_iN}{(\log R)^{{k}}}\cdot\frac{W^{{k}}}{\phi(W)^{{k+1}}},
\end{equation}
whenever $N$ is sufficiently large with respect to $w$.
\end{proposition}

It suffices to prove Proposition \ref{fK} when $i=0$.
Let $\epsilon_1=\epsilon/12$.
Choose disjoint $A_1,A_2,\ldots,A_s\in\sK$ and $a_1,\ldots,a_s>0$ such that
$$
\|f-f_*\|_{2}\leq\epsilon_1\|f\|_1,
$$
where
$$
f_*(x):=\sum_{i=1}^s a_i\cdot\1_{A_i}(x).
$$

Suppose that $B_1,B_2,\ldots,B_s\in\sB_{\cG}$ and $\psi^{-1}(B_i)=A_i$ for each $1\leq i\leq s$.  Let
$$
g_*(x):=\sum_{i=1}^s \alpha_i\cdot\1_{B_i}(x).
$$
Then $f=g_*\circ\psi$. Since $\nu=\mu\circ\psi^{-1}$, we have $\|f_*\|_2=\|g_*\|_2$,
as well as $\|f_*\|_1=\|g_*\|_1$.
We shall approximate those $B_1,\ldots,B_s$ by using some small $d$-cubes on $\T^d$. Define the $d$-cube $$
\fC_{\bfa,\eta}=\{(t_1,\ldots,t_d):\,a_i\eta\leq t_i<(a_1+1)\eta\text{ for }i=1,\ldots,d\}.$$
where $\bfa=(a_1,\ldots,a_d)\in\Z^d$ and $0\leq a_1,\ldots,a_d\leq\eta^{-1}-1$. Trivially $\fm_{\T^d}(\fC_{\bfa,\eta})=\eta^d$.

Choose a sufficiently small constant $\eta_0>0$, $\beta_1,\ldots,\beta_t>0$ and $E_1,\ldots,E_t\in\sB_\cG$ such
that

\medskip
\noindent(a) the $L^2$-norm
$$
\|g_*-g\|_2\leq\epsilon_1\|g_*\|_1,
$$
where
$$
g(x):=\sum_{i=1}^t \beta_i\cdot\1_{E_i}(x);
$$

\medskip
\noindent(b) each $E_j$ is of the form $E_j=(\gamma_j,\fC_{\bfa_j,\eta_0})$, where $\gamma_j\in H$. \medskip

\noindent Thus we have
\begin{align*}
\langle f,T^{n}f\rangle\geq&\langle f_*,T^{n}f_*\rangle-2\|f-f_*\|_2\cdot\|f\|_2-3\|f-f_*\|_2^2\\
=&\langle f_*,T^{n}f_*\rangle-3\epsilon_1=\langle g_*,S_\alpha^{n}g_*\rangle-3\epsilon_1\geq \langle g,S_\alpha^{n}g\rangle-6\epsilon_1.
\end{align*}
for any $n\in\N$. That is,
\begin{align*}
&\sum_{\substack{n\sim N\\ 
n\equiv b\pmod{W}}}\varpi(n+h_0)\Omega_n\cdot\langle f,T^{n+h_0-1}f\rangle\\
\geq&\sum_{\substack{n\sim N\\
n\equiv b\pmod{W}}}\varpi(n+h_0)\Omega_n\cdot\langle g,S_\alpha^{n+h_0-1}g\rangle-6\epsilon_1\sum_{\substack{n\sim N\\ 
n\equiv b\pmod{W}}}\varpi(n+h_0)\Omega_n.
\end{align*}

Since $S_\alpha$ is ergodic rotation on $\cG=G\oplus\T^d$, we must have $$\alpha=(\gamma_0,\kappa_1,\ldots,\kappa_d),$$ 
where $\gamma_0\in G$ and $\kappa_1,\ldots,\kappa_d\in\T$ are $\Q$-linearly independent.
Let $L_0\in\N$ be the least integer such that
$$
\bigg(1-\frac1{L_0}\bigg)^d>\bigg(1-\frac{\epsilon_1}{\|g\|_1^2}\bigg)^{\frac13}.
$$

Define the rotation $S_{\kappa_1,\ldots,\kappa_d}$ over $\T^d$ by $S_{\kappa_1,\ldots,\kappa_d}(x_1,\ldots,x_d)=(x_1+\kappa_1,\ldots,x_d+\kappa_d)$.
\begin{lemma} Suppose that $\bfa=(a_1,\ldots,a_d)$ and $\bfb=(b_1,\ldots,b_d)$, where $a_1,\ldots,a_d,b_1,\ldots,b_d$ are integers with $0\leq a_1,\ldots,a_d,b_1,\ldots,b_d\leq\eta_0^{-1}-1$. Then
\begin{align}
&\sum_{\substack{n\sim N\\ 
n\equiv b\pmod{W}}}\varpi(n+h_0)\Omega_n\cdot\fm_{\T^d}\big(\fC_{\bfa,\eta_0}\cap S_{\kappa_1,\cdots,\kappa_d}^{-(n+h_0-1)}\fC_{\bfb,\eta_0}\big)\notag\\
\geq&\bigg(1-\frac{1}{L_0}\bigg)^{3d}\cdot\eta_0^{2d}\cdot\frac{\cJ_0N}{(\log R)^{{k}}}\cdot\frac{W^{{k}}}{\phi(W)^{{k+1}}},
\end{align}
whenever $N$ is sufficiently large with respect to $w$.
\end{lemma}
\begin{proof} Let
$$
\delta_0=\frac{\eta_0}{ L_0}.
$$
For $2- L_0\leq u_1,u_2,\ldots,u_d\leq L_0-2$, 
for each $1\leq i\leq d$, it is easy to check that 
$$
m_{\T^d}\big(\fC_{\bfa,\eta_0}\cap S_{\kappa_1,\cdots,\kappa_d}^{-n}\fC_{\bfb,\eta_0}\big)\geq 
\delta_0^d\prod_{i=1}^d(L_0-|u_i|-\1_{u_i\geq 0}),
$$
provided that
\begin{equation}\label{aibinkappa}
a_i\eta_0+u_i\delta_0\leq |b_i\eta_0-n\kappa_i|_\T\leq a_i\eta_0+(u_i+1)\delta_0.
\end{equation}

Let
$$
\delta_1=\frac{\delta_0}{L_0}.
$$
Let
$\psi(x)$ be a smooth function on $\R$ with the period $1$ such that\medskip

\noindent(1) $0\leq \psi(x)\leq 1$ for any $x\in\R$;\medskip

\noindent(2) $\psi(x)=1$ if $\delta_1\leq x\leq \delta_0-\delta_1$, 
and $\psi(x)=0$ if $\delta_0\leq x<1$;
\medskip

\noindent(3) $\psi(x)$ has the Fourier expansion
$$
\psi(x)=(\delta_0-\delta_1)+\sum_{|j|\geq 1}\alpha_je(jx),
$$
where
$$
\alpha_j\ll C_0\cdot\min\{|j|^{-1},\ \delta_0-\delta_1,\ \delta_1^{-1}|j|^{-2}\}
$$
for some constant $C_0>0$.

It is well-known that such $\psi(x)$ always exists (cf. \cite[Lemma 12 of Chapter I]{Vi54}).
Let $K_0\in\N$ be the least integer such that
$$
\frac{(\log K_0+\delta_1^{-1}K_0^{-1}+1)^{d-1}}{K_0}<\frac{(\delta_0-\delta_1)^d\delta_1}{2^dC_0^dd}\cdot\bigg(1-\bigg(1-\frac1{L_0}\bigg)^{d}\bigg).
$$
Fix $u_1,u_2,\ldots,u_d$ with $|u_1|,\ldots,|u_d|\leq L_0-2$. 
Let $\Delta_i=(b_i-a_i)\eta_0-(u_i+1)\delta_0$.
Clearly $\psi\big((n+h_0-1)\kappa_i-\Delta_i\big)>0$ implies (\ref{aibinkappa}) holds.
We have
\begin{align*}
&\prod_{i=1}^d\psi\big((n+h_0-1)\kappa_i-\Delta_i\big)\\
=&\prod_{i=1}^d\bigg((\delta_0-\delta_1)+\sum_{|j|\geq 1}\alpha_je\big(j((n+h_0-1)\kappa_i-\Delta_i)\big)\bigg)\\
=&(\delta_0-\delta_1)^d+\sum_{\emptyset\neq I\subseteq\{1,\ldots,d\}}
\prod_{i\in I}\bigg(\sum_{1\leq |j|\leq K_0}\alpha_je\big(j((n+h_0-1)\kappa_i-\Delta_i)\big)\bigg)+\cR,
\end{align*}
where
\begin{align*}
|\cR|\leq&d\cdot\bigg(\sum_{|j|>K_0}\frac{C_0}{\delta_1j^2}\bigg)\cdot
\bigg((\delta_0-\delta_1)+\sum_{1\leq |j|\leq K_0}\frac{C_0}{j}+\sum_{|j|>K_0}\frac{C_0}{\delta_1j^2}\bigg)^{d-1}\\
\leq&\frac{2^dC_0^dd(\log K_0+\delta_1^{-1}K_0^{-1}+1)^{d-1}}{\delta_1K_0}<(\delta_0-\delta_1)^d\cdot\bigg(1-\bigg(1-\frac1{L_0}\bigg)^{d}\bigg).
\end{align*}
Hence
\begin{align*}
&\sum_{\substack{n\sim N\\ 
n\equiv b\pmod{W}}}\varpi(n+h_0)\Omega_n\prod_{i=1}^d\psi\big((n+h_0-1)\kappa_i-\Delta_i\big)\\
\geq&(\delta_0-\delta_1)^d\sum_{\substack{n\sim N\\ 
n\equiv b\pmod{W}}}\varpi(n+h_0)\Omega_n-|\cR|\sum_{\substack{n\sim N\\ 
n\equiv b\pmod{W}}}\varpi(n+h_0)\Omega_n\\
&-\sum_{\substack{\emptyset\neq I\subseteq\{1,\ldots,d\}\\
}}
\sum_{\substack{i\in I\\ 1\leq |j_i|\leq K_0}}\prod_{i\in I}|\alpha_{j_i}|\cdot\bigg|\sum_{\substack{n\sim N\\ 
n\equiv b\pmod{W}}}\varpi(n+h_0)\Omega_n\cdot e\bigg(n\sum_{i\in I}j_i\kappa_i\bigg)\bigg|.
\end{align*}

Assume that $I=\{i_1,\ldots,i_l\}$ is a non-empty subset of $\{1,\ldots,d\}$ and $j_{i_1},\ldots,j_{i_l}$ are integers with $1\leq |j_{i_1}|,\ldots,|j_{i_l}|\leq K_0$.
Let
$$
\vartheta_{I,j_{i_1},\ldots,j_{i_l}}=\sum_{i\in I}j_i\kappa_i.
$$ 
Since $\kappa_1,\ldots,\kappa_d$ are $\Q$-linearly independent, $\vartheta_{I,j_{i_1},\ldots,j_{i_l}}$ is also irrational. For any $x\geq 1$, there exists $1\leq a\leq q\leq x$ with $(a,q)=1$ such that
$$
\bigg|\vartheta_{I,j_{i_1},\ldots,j_{i_l}}-\frac aq\bigg|_\T\leq\frac 1{qx}.
$$
Let $\varrho_{I,j_{i_1},\ldots,j_{i_l}}(x)$ be the least one of such $q$'s. Clearly $\varrho_{I,j_{i_1},\ldots,j_{i_l}}(x)$ is an increasing function tending to $\infty$ as $x\to+\infty$. Define
$$
\varrho(x)=\min_{\substack{\emptyset\neq I\subseteq\{1,\ldots,d\}\\
I=\{i_1,\ldots,i_l\}\\ 1\leq |j_{i_1}|,\ldots,|j_{i_l}|\leq K_0}}\varrho_{I,j_{i_1},\ldots,j_{i_l}}(x)
$$

Assume that $N$ is sufficiently large such that
$$w<\frac13\log \varrho(N),$$
i.e., $W<\varrho(N)$. 
Let $P=N^{\frac13-\frac1{99}}$ and $Q=N^{1-\frac1{49}}$. 
For any $\vartheta_{I,j_{i_1},\ldots,j_{i_l}}$, there exists $\varrho(N)\leq q\leq Q$ and $1\leq a\leq q$ with $(a,q)=1$ such that
$$
\bigg|\vartheta_{I,j_{i_1},\ldots,j_{i_l}}-\frac aq\bigg|_\T\leq\frac1{qQ}.
$$
If $q\geq P$, according to (\ref{minorN}),
$$
\sum_{\substack{n\sim N\\ 
n\equiv b\pmod{W}}}\varpi(n+h_0)e(\vartheta_{I,j_{i_1},\ldots,j_{i_l}} n)\cdot\Omega_n\ll N^{1-\frac1{999}}.
$$
Suppose that $\varrho(N)\leq q<P$. Since $q\nmid W$.
using (\ref{pinhiaqtheta2}), we also have
\begin{align*}
\sum_{\substack{n\sim N\\ 
n\equiv b\pmod{W}}}\varpi(n+h_0)\Omega_n\cdot e(n\vartheta_{I,j_{i_1},\ldots,j_{i_l}})
=o_w\bigg(\frac{N}{(\log R)^{{k}}}\cdot\frac{W^{{k}}}{\phi(W)^{{k+1}}}\bigg).
\end{align*}
Therefore, in view of (\ref{pinhisum}), we get
\begin{align*}
&\sum_{\substack{n\sim N\\ 
n\equiv b\pmod{W}}}\varpi(n+h_0)\Omega_n\prod_{i=1}^d\psi\big((n+h_0-1)\kappa_i-\Delta_i\big)\\
\geq&\big((\delta_0-\delta_1)^d-|\cR|\big)\sum_{\substack{n\sim N\\ 
n\equiv b\pmod{W}}}\varpi(n+h_0)\Omega_n+o_w\bigg(\frac{N}{(\log R)^{{k}}}\cdot\frac{W^{{k}}}{\phi(W)^{{k+1}}}\bigg)\\
\geq&\bigg(1-\frac1{L_0}\bigg)^{d}\cdot(\delta_0-\delta_1)^d\cdot\frac{\cJ_0N}{(\log R)^{{k}}}\cdot\frac{W^{{k}}}{\phi(W)^{{k+1}}}\\
=&\bigg(1-\frac1{L_0}\bigg)^{2d}\cdot\delta_0^d\cdot\frac{\cJ_0N}{(\log R)^{{k}}}\cdot\frac{W^{{k}}}{\phi(W)^{{k+1}}},
\end{align*}
provided that $w$ is sufficiently large.

On the other hand,
$$
\sum_{|u_1|,\ldots,|u_d|\leq L_0-2}\prod_{i=1}^d(L_0-|u_i|-\1_{u_i\geq 0})=\bigg(\sum_{u=2-L_0}^{L_0-2}(L_0-|u_i|-\1_{u\geq 0})\bigg)^d=L_0^d(L_0-1)^d.
$$
Thus
\begin{align*}
&\sum_{\substack{n\sim N\\ 
n\equiv b\pmod{W}}}\varpi(n+h_0)\Omega_n\cdot \fm_{\T^d}\big(\fC_{\bfa,\eta_0}\cap S_{\kappa_1,\cdots,\kappa_d}^{-(n+h_0-1)}\fC_{\bfb,\eta_0}\big)\\
\geq&\delta_0^d\sum_{|u_1|,\ldots,|u_d|\leq L_0-2}\prod_{i=1}^d(L_0-|u_i|-\1_{u_i\geq 0})\sum_{\substack{n\sim N\\ 
n\equiv b\pmod{W}}}\varpi(n+h_0)\Omega_n\prod_{i=1}^d\psi\big((n+h_0-1)\kappa_i-\Delta_i\big)\\
\geq&\delta_0^d\cdot L_0^d(L_0-1)^d\cdot \bigg(1-\frac1{L_0}\bigg)^{2d}\cdot\delta_0^d\cdot\frac{\cJ_0N}{(\log R)^{{k}}}\cdot\frac{W^{{k}}}{\phi(W)^{{k+1}}}\\
=&\bigg(1-\frac1{L_0}\bigg)^{3d}\cdot\eta_0^{2d}\cdot\frac{\cJ_0N}{(\log R)^{{k}}}\cdot\frac{W^{{k}}}{\phi(W)^{{k+1}}}.
\end{align*}
\end{proof}

Recall that $E_i=(\gamma_i,\fC_{\bfa_i,\eta_0})$ where $\gamma_i\in G$.
For any $\gamma\in G$, let $V_\gamma=\{1\leq i\leq t:\, \gamma_i=\gamma\}$. 
Since $W_0$ is a multiple of $|G|$, according to the assumptions (I)-(IV) in Section 2, 
$n+h_0-1,n+h_1-1,\ldots,n+h_k-1$ are all divisible by $|G|$ whenever $n\equiv b\pmod{W}$.
So for any $1\leq i,j\leq t$,
$$
\nu(E_i\cap S_{\alpha}^{-(n+h_0-1)}E_j)=\frac1{|G|}\cdot\fm_{\T^d}\big(\fC_{\bfa_i,\eta_0}\cap S_{\kappa_1,\cdots,\kappa_d}^{-(n+h_0-1)}\fC_{\bfa_j,\eta_0}\big)
$$
if and only if $i,j$ belong to the same $V_\gamma$.
We have
\begin{align*}
&\sum_{\substack{n\sim N\\
n\equiv b\pmod{W}}}\varpi(n+h_0)\Omega_n\cdot\langle g,S_\alpha^{n+h_0-1}g\rangle\\
=&\sum_{\substack{n\sim N\\
n\equiv b\pmod{W}}}\varpi(n+h_0)\Omega_n\cdot\sum_{1\leq i,j\leq t}\beta_i\beta_j\cdot\nu(E_i\cap S_{\alpha}^{-(n+h_0-1)}E_j)\\
=&\sum_{\gamma\in G}\frac1{|G|}\sum_{i,j\in V_\gamma}\beta_i\beta_j\sum_{\substack{n\sim N\\
n\equiv b\pmod{W}}}\varpi(n+h_0)\Omega_n\cdot\fm_{\T^d}\big(\fC_{\bfa_i,\eta_0}\cap S_{\kappa_1,\cdots,\kappa_d}^{-(n+h_0-1)}\fC_{\bfa_j,\eta_0}\big)\\
\geq&\bigg(1-\frac1{L_0}\bigg)^{3d}\cdot\frac{\cJ_0N}{(\log R)^{{k}}}\cdot\frac{W^{{k}}}{\phi(W)^{{k+1}}}\cdot\frac{\eta_0^{2d}}{|G|}\sum_{\gamma\in G}\sum_{i,j\in V_\gamma}\beta_i\beta_j.
\end{align*}
By the Cauchy-Schwarz inequality,
\begin{align*}
\|g\|_1^2=&
\bigg(\sum_{\gamma\in G}\frac{\eta_0^d}{|G|}\sum_{i\in V_\gamma}\beta_i\bigg)^2\\
\leq&\bigg(\sum_{\gamma\in G}\frac{1}{|G|}\bigg)\cdot\bigg(\sum_{\gamma\in G}\frac{\eta_0^{2d}}{|G|}\bigg(\sum_{i\in V_\gamma}\beta_i\bigg)^2\bigg)
=\frac{\eta_0^{2d}}{|G|}\sum_{\gamma\in H}\sum_{i,j\in V_\gamma}\beta_i\beta_j.
\end{align*}
And since $\|f_*\|_1=\|g_*\|_1$, we have
\begin{align*}
\big|\|f\|_1-\|g\|_1\big|\leq\|f-f_*\|_1+\|g_*-g\|_1
\leq(2\epsilon_1+\epsilon_1^2)\|f\|_1,
\end{align*}
i.e.,
$$
\|g\|_1^2\geq(1-2\epsilon_1-\epsilon_1^2)^2\|f\|_1^2\geq\|f\|_1^2-4\epsilon_1.
$$

When $w$ is sufficiently large, we have
\begin{align*}
&\sum_{\substack{n\sim N\\
n\equiv b\pmod{W}}}\varpi(n+h_0)\Omega_n\cdot\langle f,T^{n+h_0-1}f\rangle\\
\geq&\sum_{\substack{n\sim N\\
n\equiv b\pmod{W}}}\varpi(n+h_0)\Omega_n\cdot\langle g,S_\alpha^{n+h_0-1}g\rangle-6\epsilon_1\sum_{\substack{n\sim N\\
n\equiv b\pmod{W}}}\varpi(n+h_0)\Omega_n\\
\geq&\bigg(\|g\|_1^2\cdot\bigg(1-\frac1{L_0}\bigg)^{3d}-7\epsilon_1\bigg)\cdot\frac{\cJ_0N}{(\log R)^{{k}}}\cdot\frac{W^{{k}}}{\phi(W)^{{k+1}}}\\
\geq&\big(\|f\|_1^2-12\epsilon_1\big)\cdot\frac{\cJ_0N}{(\log R)^{{k}}}\cdot\frac{W^{{k}}}{\phi(W)^{{k+1}}}.
\end{align*}

\section{The proof of Theorem \ref{mainT}}
\setcounter{lemma}{0}\setcounter{theorem}{0}\setcounter{proposition}{0}\setcounter{corollary}{0}\setcounter{remark}{0}
\setcounter{equation}{0}

We shall complete the proof of Theorem \ref{mainT}.
\begin{proposition}\label{SumTnhiT} Suppose that $A\subseteq\sB_\cX$ with $\mu(A)>0$ and $0<\epsilon<\mu(A)^2$.
There exist integers $W,b>0$ with $(b,W)=1$  and an adimissible set $\{h_0,\ldots,h_k\}$ such that
\begin{equation}
\sum_{\substack{n\sim N\\
n\equiv b\pmod{W}}}\varpi(n+h_i)\Omega_n\cdot\mu(A\cap T^{-(n+h_i-1)}A)\geq(\mu(A)^2-\epsilon)\cdot\frac{\cJ_iN}{(\log R)^{{k}}}\cdot\frac{W^{{k}}}{\phi(W)^{{k+1}}}
\end{equation}
for any sufficiently large $N$ and each $0\leq i\leq k$.
\end{proposition}
First, assume that $T$ is ergodic.
Suppose that $(\cG,\sB_\cG,\nu,S_\alpha)$ is the Kronecker factor of $(\cX,\sB_\cX,\mu,T)$. Write $\cG=G_{\mu}\oplus\T^d$
where $G_{\mu}$ is a finite abelian group. 

Suppose that $\{h_0,\ldots,h_k\}$ is an admissible set with $h_0,\ldots,h_k$ are all the multiple of $|G_{\mu}|$. Also, assume that $W_0$ is a positive integer divisible by $|G_{\mu}|$. 
\begin{lemma}\label{ergodicTnhiL} Suppose that $f\in L_\mu^2$ is a non-negative function. Then for any $\epsilon>0$, there exist $w_\mu(\epsilon)>0$ such that for any $w\geq w_\mu(\epsilon)$ and $0\leq i\leq k$,
\begin{equation}\label{ergodicTnhi}
\sum_{\substack{n\sim N\\
n\equiv b\pmod{W}}}\varpi(n+h_i)\Omega_n\cdot\langle f,T^{n+h_i-1}f\rangle\geq(\|f\|_1^2-\epsilon)\cdot\frac{\cJ_iN}{(\log R)^{{k}}}\cdot\frac{W^{{k}}}{\phi(W)^{{k+1}}}
\end{equation}
provided that $N$ is sufficiently large, where
$
W=W_0\prod_{p\leq w}p
$.
\end{lemma}
\begin{proof}
This is the immediate consequence of Propositions \ref{fK0} and \ref{fK}, by noting that $\big\|\E(f|\cK)\big\|_1=\|f\|_1$.
\end{proof}

Assume that $T$ is not necessarily ergodic. According to the ergodic decomposition theorem (cf. \cite[Theorem 6.2]{EW11}), there exists a probability space $(\cY,\sB_\cY,\upsilon)$ such that
$$
\mu=\int_{\cY}\mu_yd\upsilon(y),
$$
and $\mu_y$ is a $T$-invariant ergodic measure on $\cX$ for almost every $y$. 
Let $\cY_0$ denote the set of all $y\in\cY$ such that $T$ is ergodic with respect to $\mu_y$.

Let $\epsilon_1=\epsilon/5$. For any given positive integers $w_0$ and $g_0$, let
$$
\cU_{w_0,g_0}=\{y\in\cY_0:\,w_{\mu_y}(\epsilon_1)=w_0,\ |G_{\nu_y}|=g_0\}.
$$
Note that
$$
1=\mu(\cX)=\int_{\cY}\mu_y(\cX)d\upsilon(y)=\sum_{s,t\geq 1}\int_{\cU(s,t)}1d\upsilon(y).
$$
So we may choose sufficiently large $w_0$ and $g_0$ such that
$$
\sum_{s=1}^{w_0}\sum_{t=1}^{g_0}\int_{\cU(s,t)}1d\upsilon(y)\geq1-\epsilon_1.
$$
Let
$$
\cY_1=\bigcup_{\substack{1\leq s\leq w_0\\ 1\leq t\leq g_0}}\cU_{s,t}.
$$
Clearly $\upsilon(\cY\setminus\cY_1)\leq \epsilon_1$.

Let $W_0=[1,2,\ldots,g_0]$ and
$$
W=W_0\prod_{p\leq w_0}p.
$$
Choose $b,h_0,\ldots,h_k$ satisfying the assumptions (I)-(IV) in Section 2.
Furthermore, we may assume that $w_0$ is sufficiently large such that
$$
\sum_{\substack{n\sim N\\
n\equiv b\pmod{W}}}\varpi(n+h_i)\Omega_n\leq \frac{2\cJ_iN}{(\log R)^{{k}}}\cdot\frac{W^{{k}}}{\phi(W)^{{k+1}}}
$$

According to Lemma \ref{ergodicTnhiL}, we know that for any $y\in\cY_1$, there exists $N_{\mu_y}>0$ such that
\begin{equation}\label{TnhiMuy}
\sum_{\substack{n\sim N\\
n\equiv b\pmod{W}}}\varpi(n+h_i)\Omega_n\cdot\mu_y(A\cap T^{-(n+h_i-1)}A)\geq(\mu_y(A)^2-\epsilon_1)\cdot\frac{\cJ_iN}{(\log R)^{{k}}}\cdot\frac{W^{{k}}}{\phi(W)^{{k+1}}}
\end{equation}
for any $N\geq N_{\mu_y}$. 
For any positive $N_0$, let
$$
\cV(N_0)=\{y\in\cY_1:\,N_{\mu_y}=N_0\}.
$$
Then we may choose a large $N_0$ such that
$$
\sum_{r>N_0}\int_{\cV(r)}1d\upsilon(y)\leq \epsilon_1.
$$
Let
$$
\cY_2=\bigcup_{1\leq r\leq N_0}\cV(r).
$$
Then for any $A_*\in\sB_\cX$, we have
\begin{align*}
\mu(A_*)-\int_{\cY_2}\mu_y(A_*)d\upsilon(y)
=\int_{\cY\setminus\cY_2}\mu_y(A_*)d\upsilon(y)
\leq\upsilon(\cY\setminus\cY_1)+\upsilon(\cY_1\setminus\cY_2)\leq2\epsilon_1.
\end{align*}
Thus (\ref{TnhiMuy}) is valid for any $y\in\cY_2$ and $N>N_0$,
So for any $N>N_0$
\begin{align*}
&\sum_{\substack{n\sim N\\
n\equiv b\pmod{W}}}\varpi(n+h_i)\Omega_n\int_{\cY_2}
\mu_y(A\cap T^{-(n+h_i-1)}A)d\upsilon(y)\\
\geq& 
\frac{\cJ_iN}{(\log R)^{{k}}}\cdot\frac{W^{{k}}}{\phi(W)^{{k+1}}}\cdot\bigg(\int_{\cY_2}\mu_y(A)^2d\upsilon(y)-\epsilon_1\bigg).
\end{align*}
By the Cauchy-Schwarz inequality,
\begin{align*}
\int_{\cY_2}\mu_y(A)^2d\upsilon(y)\geq&\frac1{\upsilon(\cY_2)}\bigg(\int_{\cY_2}\mu_y(A)d\upsilon(y)\bigg)^2
\\
\geq&\frac1{\upsilon(\cY_2)}\bigg(\int_{\cY}\mu_y(A)d\upsilon(y)-2\epsilon_1\bigg)^2\geq 
\frac{\mu(A)^2-4\epsilon_1}{\upsilon(\cY_2)}.
\end{align*}
Hence
\begin{align*}
&\sum_{\substack{n\sim N\\
n\equiv b\pmod{W}}}\varpi(n+h_i)\Omega_n\mu(A\cap T^{-(n+h_i-1)}A)\\
\geq&\sum_{\substack{n\sim N\\
n\equiv b\pmod{W}}}\varpi(n+h_i)\Omega_n
\int_{\cY_2}
\mu_y(A\cap T^{-(n+h_i-1)}A)d\upsilon(y)\\
\geq&\frac{\cJ_iN}{(\log R)^{{k}}}\cdot\frac{W^{{k}}}{\phi(W)^{{k+1}}}\bigg(\frac{\mu(A)^2-4\epsilon_1}{\upsilon(\cY_2)}d\upsilon(y)-\epsilon_1\bigg)\\
\geq&(\mu(A)^2-5\epsilon_1)\cdot\frac{\cJ_iN}{(\log R)^{{k}}}\cdot\frac{W^{{k}}}{\phi(W)^{{k+1}}}.
\end{align*}
Proposition \ref{SumTnhiT} is concluded.\qed

By Proposition \ref{SumTnhiT}, we may choose $w_0,W_0,b,h_0,\ldots,h_k$ such that
$$
\sum_{\substack{n\sim N\\
n\equiv b\pmod{W}}}\varpi(n+h_i)\Omega_n\cdot\mu(A\cap T^{-(n+h_i-1)}A)\geq\bigg(\mu(A)^2-\frac\epsilon4\bigg)\cdot\frac{\cJ_iN}{(\log R)^{{k}}}\cdot\frac{W^{{k}}}{\phi(W)^{{k+1}}}
$$
for any sufficiently large $N$ and each $0\leq i\leq k$, where $W=W_0\prod_{p\leq w_0}p$.
Furthermore, we may assume that
$$
\sum_{\substack{n\sim N\\
n\equiv b\pmod{W}}}\Omega_n\leq\frac{2\cJ_*N}{(\log R)^{{k+1}}}\cdot\frac{W^{{k}}}{\phi(W)^{{k+1}}}.
$$
and
$$
\sum_{\substack{n\sim N\\
n\equiv b\pmod{W}}}\varpi(n+h_i)\Omega_n\leq
\bigg(1+\frac\epsilon4\bigg)\cdot\frac{\cJ_iN}{(\log R)^{{k}}}\cdot\frac{W^{{k}}}{\phi(W)^{{k+1}}}.
$$
for each $0\leq i\leq k$.
Then
\begin{align*}
&\sum_{\substack{n\sim N\\
n\equiv b\pmod{W}}}\varpi(n+h_i)\Omega_n\cdot\big(\mu(A\cap T^{-(n+h_i-1)}A)-(\mu(A)^2-\epsilon)\big)
\\\geq&\bigg(\mu(A)^2-\frac\epsilon4\bigg)\cdot\frac{\cJ_iN}{(\log R)^{{k}}}\cdot\frac{W^{{k}}}{\phi(W)^{{k+1}}}-(\mu(A)^2-\epsilon)\bigg(1+\frac\epsilon4\bigg)\cdot\frac{\cJ_iN}{(\log R)^{{k}}}\cdot\frac{W^{{k}}}{\phi(W)^{{k+1}}}\\
\geq&\frac\epsilon2\cdot\frac{\cJ_iN}{(\log R)^{{k}}}\cdot\frac{W^{{k}}}{\phi(W)^{{k+1}}}.
\end{align*}

According the the discussions in \cite{Ma15} and \cite{T}, for any $m\geq 1$, we may choose a sufficiently large $k$ and construct a symmetric smooth function $F(t_0,\ldots,t_k)$ such that
$$
\cJ_0(F)=\cdots=\cJ_k(F)\geq \frac{10^4m}{k\epsilon}\cdot \cJ_*(F).
$$
Thus
\begin{align*}
&\sum_{\substack{n\sim N\\
n\equiv b\pmod{W}}}\Omega_n\bigg(\sum_{i=0}^k\varpi(n+h_i)\big(\mu(A\cap T^{-(n+h_i-1)}A)-(\mu(A)^2-\epsilon)\big)-m\log 3N\bigg)
\\\geq&k\cdot\frac\epsilon2\cdot\frac{\cJ_0N}{(\log R)^{{k}}}\cdot\frac{W^{{k}}}{\phi(W)^{{k+1}}}-\frac{2\cJ_*N\log N}{(\log R)^{{k+1}}}\cdot\frac{W^{{k}}}{\phi(W)^{{k+1}}}>0.
\end{align*}
Therefore there exists $N\leq n\leq 2N$ such that
$$
\sum_{i=0}^k\varpi(n+h_i)\big(\mu(A\cap T^{-(n+h_i-1)}A)-(\mu(A)^2-\epsilon)\big)-m\log 3N>0.
$$
That is, there exist $0\leq i_0<i_1<\cdots<i_m\leq k$ such that
$n+h_{i_0},n+h_{i_1},\ldots,n+h_{i_m}$ are all primes and
$$
\mu(A\cap T^{-(n+h_{i_j}-1)}A)>\mu(A)^2-\epsilon
$$
for each $0\leq j\leq m$.
Evidently the gaps between those primes are bounded by $\max_{0\leq i<j\leq k}|h_j-h_i|$.

\begin{remark}
In fact, we may require that the primes $p_0,p_2,\ldots,p_m$ in Theorem \ref{mainT} are consecutive primes. Assume that $0\leq h_0<h_2<\cdots<h_k$ and write
$$
\{a_1,a_2,\ldots,a_{h_k-h_0-k}\}=\{h_0,h_0+1,\ldots,h_k-1,h_k\}\setminus\{h_0,h_1,\ldots,h_k\}.
$$
Arbitrarily choose distinct primes $h_k<\rho_1,\rho_2,\ldots,\rho_{h_k-h_0-k}\leq w$, which doesn't divide $W_0$.
Let $b$ satisfy an additional requirement:
$$
b\equiv -a_j\pmod{\rho_j}
$$
for each $1\leq j\leq h_k-h_0-k$. 
Evidently the above assumption would not contradict with (\ref{bhjp}), since
$\rho_j>|h_i-a_j|$ for any $0\leq i\leq k$. Thus if $n\sim N$ and $n\equiv b\pmod{W}$,
for each $1\leq j\leq h_k-h_0-k$, $n+a_j$ can not be a prime since it is divisible by $\rho_j$. It follows that those primes among $n+h_0,n+h_1,\ldots,n+h_k$ are surely consecutive primes.
\end{remark}
\begin{remark}
In \cite{BHK05}, Bergelson, Host and Kra considered the generalizations of 
Khintchine's theorem for multiple recurrence problems. Suppose that 
$(\cX,\sB_\cX,\mu,T)$ is a measure-preserving probability system and $T$ is an invertible ergodic transformation.
They proved that for any $A\in\sB_\cX$ with $\mu(A)>0$ and $\epsilon>0$, the sets of recurrence
$$
\{n\in\N:\,\mu(A\cap T^{-n}A\cap T^{-2n}A)\geq \mu(A)^3-\epsilon\}
$$
and
$$
\{n\in\N:\,\mu(A\cap T^{-n}A\cap T^{-2n}A\cap T^{-3n}A)\geq \mu(A)^4-\epsilon\}
$$
both have the bounded gaps. However, in general,
$$
\{n\in\N:\,\mu(A\cap T^{-n}A\cap T^{-2n}A\cap T^{-3n}A\cap T^{-4n}A)\geq \mu(A)^5-\epsilon\}
$$
doesn't have a bounded gap.

Let
$$
\Lambda_{A,\epsilon}^{(2)}=\{p\text{ prime}:\,\mu(A\cap T^{-(p-1)}A\cap T^{-2(p-1)}A)\geq \mu(A)^3-\epsilon\}
$$
and
$$
\Lambda_{A,\epsilon}^{(3)}=\{p\text{ prime}:\,\mu(A\cap T^{-(p-1)}A\cap T^{-2(p-1)}A\cap T^{-3(p-1)}A)\geq \mu(A)^4-\epsilon\}.
$$
Naturally, we may ask whether the Maynard-Tao theorem can be extended to $\Lambda_{A,\epsilon}^{(2)}$ and $\Lambda_{A,\epsilon}^{(3)}$.
Unfortunately, the exponential sums and the Kronecker factors are not sufficient to attack this problem. According to the work of Host and Kra \cite{HK05}, in order to deal with the multiple recurrence problems, we need to use the characters on the nilmanifolds (cf. \cite{GT12}) and the factors arising from the Host-Kra semi-norms.
However, the main difficulty is how to establish a suitable mean-value type summation formula and combine this formula with the Maynard sieve method.   
\end{remark}


\begin{thebibliography}{99}

\bibitem{BP85} A. Balog and A. Perelli, {\it Exponential sums over primes in an arithmetic progression}, Proc. Amer. Math. Soc., {\bf 93}(1985), 578-582.

\bibitem{BHK05} V. Bergelson, B. Host and B. Kra, \textit{Multiple recurrence and
nilsequences, with an Appendix by I. Ruzsa},  Invent. Math., {\bf 160}(2005) 261–303.

\bibitem{BL08} V. Bergelson and E. Lesigne, \textit{Van der Corput sets in $\Z^d$}, Colloq. Math., {\bf 110}(2008), 1-49.

\bibitem{Ch11} Q. Chu, \textit{Multiple recurrence for two commuting transformations}, Ergodic Theory Dynam. Systems, {\bf 31}(2011), 771-792.

\bibitem{CFH11}
Q. Chu, N. Frantzikinakis and B. Host, \textit{Ergodic averages of commuting transformations with distinct degree polynomials iterates}, Proc. London Math. Soc., {\bf 102}(2011), 801-842. 

\bibitem{CPS15} L. Chua, S. Park and G. D. Smith, {\it Bounded gaps between primes in special sequences}, Proc. Amer. Math. Soc., {\bf 143}(2015), 4597-4611.

\bibitem{EW11} M. Einsiedler, Manfred and T. Ward, {\it Ergodic theory with a view towards number theory}, Graduate Texts in Mathematics, 259, Springer-Verlag London, Ltd., London, 2011.

\bibitem{Fr08} N. Frantzikinakis, \textit{Multiple ergodic averages for three polynomials and applications}, Trans. Amer. Math. Soc., {\bf 360}(2008), 5435–5475.

\bibitem{Fu81} H. Furstenberg, {\it Poincar\'e recurrence and number theory}, Bull. Amer. Math. Soc. (N.S.), {\bf 5}(1981), 211-234.

\bibitem{GPY09} D. A. Goldston, J. Pintz and C. Y. Y{\i}d{\i}r{\i}m, {\it Primes in tuples. I}, Ann. of
Math. (2), {\bf 170}(2009), 819-862.

\bibitem{Ga15} A. Granville, {\it Primes in intervals of bounded length}, Bull. Amer. Math. Soc. (N.S.), {\bf 52}(2015), 171-222.

\bibitem{GT12} B. Green and T. Tao, {\it The M\"obius function is strongly orthogonal to nilsequences}, Ann. of Math. (2), {\bf 175}(2012), 541-566.

\bibitem{He83} H. Helson, {\it Harmonic analysis}, Addison-Wesley Publishing Company, Advanced Book Program, Reading, MA, 1983.

\bibitem {HK05} B. Host and B. Kra, {\it Nonconventional ergodic averages and nilmanifolds}, Ann. of Math. (2), {\bf 161}(2005), 397-488. 

\bibitem{Kh34} A. Y. Khintchine, \textit{Eine Versch\"arfung des Poincar\'eschen ``Wiederkehr-Satzes"},
Compos. Math., {\bf 1}(1934), 177-179.

\bibitem{LP15} H.-Z. Li and H. Pan, {\it Bounded gaps between primes of a special form}, Int. Math. Res. Not., 2015. 12345-12365.

\bibitem{LZ97} J. Y. Liu and T. Zhan, {\it The ternary Goldbach problem in arithmetic progressions}, Acta Arith., {\bf 82}(1997), 197-227.

\bibitem{Ma15} J. Maynard, {\it Small gaps between primes}, Ann. of Math. (2), {\bf 181}(2015), 383-413.

\bibitem{Po14} P. Pollack, {\it Bounded gaps between primes with a given primitive root}, Algebra Number Theory, {\bf 8}(2014), 1769-1786.

\bibitem{Pol14} D. H. J. Polymath, {\it Variants of the Selberg sieve, and bounded intervals containing many primes}, Res. Math. Sci., {\bf 1}(2014), Art. 12, 83 pp.

\bibitem{Sa78} A. S\'ark\"ozy, \textit{On difference sets of sequences on integers III},
Acta Math. Acad. Sci. Hungar., \textbf{31}(1978), 355-386.

\bibitem{T} T. Tao, {\it Polymath8b: Bounded intervals with many primes, after Maynard},
http://terrytao.wordpress.com/2013/11/19/polymath8b-bounded-intervals-with-many-primes-after-maynard.

\bibitem{Vi54} Vinogradov,I. M.: \textit{The method of trigonometrical sums in the theory of numbers}, translated, revised and annotated by K. F. Roth and Anne Davenport, Interscience Publishers, London and New York, 1954

\bibitem{Zh14} Y.-T. Zhang, {\it Bounded gaps between primes}, Ann. of Math. (2), {\bf 179}(2014), 1121-1174.

\end{thebibliography}
\end{document}